\documentclass[10pt,oneside,english]{amsart}

\usepackage[latin9]{inputenc}
\usepackage{amsthm}
\usepackage{amssymb}

\makeatletter
\numberwithin{equation}{section} %% Comment out for sequentially-numbered
\numberwithin{figure}{section} %% Comment out for sequentially-numbered
\theoremstyle{plain}
  \theoremstyle{plain}
  \newtheorem*{thm*}{Theorem}
\theoremstyle{plain}
\newtheorem{thm}{Theorem}
  \theoremstyle{definition}
  \newtheorem{defn}[thm]{Definition}
  \theoremstyle{plain}
  \newtheorem{lem}[thm]{Lemma}

\makeatother

\usepackage{babel}

\begin{document}

\title[traveling wave solution]{on the minimal speed and asymptotics of the wave solutions for the
lotka volterra system}

\author{xiaojie hou}

\address{Department of Mathematics and Statistics, University of North Carolina-Wilmington,
Wilmington, NC 28409}

\email{houx@uncw.edu}


\begin{abstract}
We study the minimal wave speed and the asymptotics of the traveling
wave solutions of a competitive Lotka Volterra system. The existence
of the traveling wave solutions is derived by monotone iteration.
The asymptotic behaviors of the wave solutions are derived by comparison
argument and the exponential dichotomy, which seems to be the key
to understand the geometry and the stability of the wave solutions.
Also the uniqueness and the monotonicity of the waves are investigated
via a generalized sliding domain method.
\end{abstract}
\maketitle
\bigskip{}

\section{\textbf{Introduction and the Main result\label{sec:1}}}

\setcounter{equation}{0}

We study the minimal wave speed and the asymptotic behaviors of the
traveling wave solutions of the following classical Lotka-Volterra
competition system

\begin{equation}
\left\{ \begin{array}{l}
u_{t}=u_{xx}+u(1-u-a_{1}v),\\
\\v_{t}=v_{xx}+rv(1-a_{2}u-v)\end{array}\quad\quad(x,t)\in\mathbb{R}\times\mathbb{R}^{+}\right.\label{eq:1.01}\end{equation}
 where $u=u(x,t),v=v(x,t)$ and $a_{1},\: a_{2},\: r$ are positive
constants. 

In the wave coordinates $\xi=x+ct$, \eqref{eq:1.01} is changed into 

\begin{equation}
\left\{ \begin{array}{l}
u_{\xi\xi}-cu_{\xi}+u(1-u-a_{1}v)=0,\\
\\v_{\xi\xi}-cv_{\xi}+rv(1-a_{2}u-v)=0\end{array}\quad\quad\xi\in\mathbb{R}.\right.\label{eq:1.02}\end{equation}

Fei and Carr \cite{Fei} investigated the traveling wave solutions
and their minimal wave speed of system \eqref{eq:1.02} under the
assumptions:

\medskip{}

\textbf{{[}H1{]}}.\label{[H1]}\qquad{} $0<a_{1}<1<a_{2}$, 

\medskip{}

\textbf{{[}H2{]}}.\label{[H2]} \qquad{}$1-a_{1}\leq r(a_{2}-1)$. 

\medskip{}

Requiring further $r(a_{2}-1)\leq1$, they showed that for each speed
$c\geq2\sqrt{r(a_{2}-1)}$ system \eqref{eq:1.02} admits monotonic
traveling waves $(u(\xi),v(\xi))^{T}$ satisfying the following boundary
conditions:

\begin{equation}
\left(\begin{array}{l}
u\\
v\end{array}\right)(-\infty)=\left(\begin{array}{l}
0\\
1\end{array}\right),\quad\left(\begin{array}{l}
u\\
v\end{array}\right)(+\infty)=\left(\begin{array}{l}
1\\
0\end{array}\right).\label{eq:1.03}\end{equation}

Under additional assumptions $r=1,$ $a_{1}+a_{2}=2$ or $r(a_{2}-1)=1-a_{1}$
they also showed that system \eqref{eq:1.02} has monotonic traveling
wave solutions satisfying \eqref{eq:1.03} for $c\geq2\sqrt{1-a_{1}}$.
However, the question of the minimal wave speed for the wave solutions
of \eqref{eq:1.02}-\eqref{eq:1.03} remains unanswered.

We will prove that the minimal wave speed for \eqref{eq:1.02} is
indeed $2\sqrt{1-a_{1}}$ if the following additional assumption

\smallskip{}

\textbf{{[}H3{]}.\label{[H3]}} $\qquad r(a_{2}-1)<(1-a_{1})(2-a_{1}+r)$. 

\smallskip{}

\noindent is imposed. Noting that if $r=1$ the condition \textbf{{[}H3{]}}
includes Fei and Carr's additional condition.

System \eqref{eq:1.02} has three non-negative equilibria $(0,0)$,
$(0,1)$ and $(1,0)$, with $(0,0)$ and $(0,1)$ unstable and $(1,0)$
stable (\cite{Fei}). For the convenience of later use, we introduce
the transformation $\hat{v}=1-v$ to change system \eqref{eq:1.02}
into local monotone. Upon dropping the hat on the function t $v$
system \eqref{eq:1.02} is changed into \begin{equation}
\left\{ \begin{array}{l}
u_{\xi\xi}-cu_{\xi}+u(1-a_{1}-u+a_{1}v)=0,\\
\\v_{\xi\xi}-cv_{\xi}+r(1-v)(a_{2}u-v)=0\end{array}\quad\xi\in\mathbb{R},\right.\label{eq:1.04}\end{equation}
and the boundary conditions \eqref{eq:1.03} are changed into\begin{equation}
\left(\begin{array}{l}
u\\
v\end{array}\right)(-\infty)=\left(\begin{array}{l}
0\\
0\end{array}\right),\quad\left(\begin{array}{l}
u\\
v\end{array}\right)(+\infty)=\left(\begin{array}{l}
1\\
1\end{array}\right).\label{eq:1.05}\end{equation}
We have,
\begin{thm*}
\label{thm:1}Assuming conditions \textbf{{[}H1{]}-{[}H3{]}}, then
$c^{*}=2\sqrt{1-a_{1}}$ is the minimal wave speed for system \eqref{eq:1.04}-\eqref{eq:1.05},
namely, corresponding to each fixed $c\geq c^{*}$ \eqref{eq:1.04}-\eqref{eq:1.05}
has a unique traveling wave solution $(u(\xi),v(\xi))^{T}$; while
for $0<c<c^{*}$, \eqref{eq:1.04}-\eqref{eq:1.05} does not have
any monotonic traveling wave solutions. Furthermore, the traveling
wave solution has the following asymptotic behaviors:

1. Corresponding to each wave speed $c>2\sqrt{1-a_{1}}$, the traveling
wave solution $(u(\xi),v(\xi))^{T}$ satisfies, as $\xi\rightarrow-\infty$;
\begin{equation}
\left(\begin{array}{c}
u(\xi)\\
\\v(\xi)\end{array}\right)=\left(\begin{array}{c}
A_{1}\\
\\A_{2}\end{array}\right)e^{\frac{c-\sqrt{c^{2}-4(1-a_{1})}}{2}\xi}+o(e^{\frac{c-\sqrt{c^{2}-4(1-a_{1})}}{2}\xi}),\label{eq:1.06}\end{equation}
 While as $\xi\rightarrow+\infty$ we have two cases to deal with:
if $r(a_{2}-1)\leq1$, then \begin{equation}
\left(\begin{array}{c}
u(\xi)\\
\\v(\xi)\end{array}\right)=\left(\begin{array}{c}
1\\
\\{\displaystyle 1}\end{array}\right)-\left(\begin{array}{c}
\bar{A}_{1}\\
\\\bar{A}_{2}\end{array}\right)e^{\frac{c-\sqrt{c^{2}+4r(a_{2}-1)}}{2}\xi}+o(e^{\frac{c-\sqrt{c^{2}+4r(a_{2}-1)}}{2}\xi}),\label{eq:1.07}\end{equation}
and if $r(a_{2}-1)>1$, then \begin{equation}
\left(\begin{array}{c}
u(\xi)\\
\\v(\xi)\end{array}\right)=\left(\begin{array}{c}
1\\
\\{\displaystyle 1}\end{array}\right)-\left(\begin{array}{c}
\hat{A}_{1}e^{\frac{c-\sqrt{c^{2}+4}}{2}\xi}+\hat{A}_{2}e^{\frac{c-\sqrt{c^{2}+4r(a_{2}-1)}}{2}\xi}\\
\\\hat{A}_{3}e^{\frac{c-\sqrt{c^{2}+4r(a_{2}-1)}}{2}\xi}\end{array}\right)+\left(\begin{array}{c}
o(e^{\frac{c-\sqrt{c^{2}+4}}{2}\xi})\\
\\o(e^{\frac{c-\sqrt{c^{2}+4r(a_{2}-1)}}{2}\xi})\end{array}\right)\label{eq:1.08}\end{equation}
where $A_{1}$, $A_{2}$, $\bar{A}_{1}$, $\bar{A}_{2}$, \textup{$\hat{A}_{1}$,
$\hat{A}_{3}$} are positive constants, and $\hat{A}_{2}$ is a real
number. 

2. For the traveling wave with the critical speed $c_{\mbox{ }}^{*}=2\sqrt{1-a_{1}}$,
$(u(\xi),v(\xi))^{T}$ satisfies \begin{equation}
\left(\begin{array}{c}
u(\xi)\\
\\v(\xi)\end{array}\right)=\left(\begin{array}{c}
A_{11c}+A_{12c}\xi\\
\\A_{21c}+A_{22c}\xi\end{array}\right)e^{\sqrt{1-a_{1}}\xi}+o(\xi e^{\sqrt{1-a_{1}}\xi})\label{eq:1.09}\end{equation}
as $\xi\rightarrow-\infty$, and if $r(a_{2}-1)\leq1$, we have\begin{equation}
\begin{array}{cc}
\left(\begin{array}{c}
u(\xi)\\
\\v(\xi)\end{array}\right)= & \left(\begin{array}{c}
1\\
\\{\displaystyle 1}\end{array}\right)-\left(\begin{array}{c}
\bar{A}_{11}\\
\\\bar{A}_{22}\end{array}\right)e^{(\sqrt{1-a_{1}}-\sqrt{1-a_{1}+r(a_{2}-1)})\xi}\\
\\ & +o(e^{(\sqrt{1-a_{1}}-\sqrt{1-a_{1}+r(a_{2}-1)})\xi});\end{array}\label{eq:1.10}\end{equation}

while if $r(a_{2}-1)>1$, we have\begin{equation}
\begin{array}{ccc}
\left(\begin{array}{c}
u(\xi)\\
\\v(\xi)\end{array}\right) & = & \left(\begin{array}{c}
1\\
\\{\displaystyle 1}\end{array}\right)-\left(\begin{array}{c}
\hat{A}_{11}e^{(\sqrt{1-a_{1}}-\sqrt{2-a_{1}})\xi}+\hat{A}_{12}e^{(\sqrt{1-a_{1}}-\sqrt{1-a_{1}+r(a_{2}-1)})\xi}\\
\\\hat{A}_{22}e^{(\sqrt{1-a_{1}}-\sqrt{1-a_{1}+r(a_{2}-1)})\xi}\end{array}\right)\\
\\ &  & +\left(\begin{array}{c}
o(e^{(\sqrt{1-a_{1}}-\sqrt{2-a_{1}})\xi})\\
\\o(e^{(\sqrt{1-a_{1}}-\sqrt{1-a_{1}+r(a_{2}-1)})\xi})\end{array}\right)\end{array}\label{eq:1.11}\end{equation}
 as $\xi\rightarrow+\infty$, where $A_{12c},\, A_{22c}<0$, $A_{11c}$,
$A_{21c}\in\mathbb{R}$ and $\bar{A}_{11},$$\bar{A}_{22}$, $\hat{A}_{12}$,
$\hat{A}_{22}>0$, $\hat{A}_{11}$ is a real number. 
\end{thm*}
\noindent \vspace{7pt}

In the next section we prove the theorem. The proof uses monotone
iteration of a pair of upper and lower solutions, which is different
from that of {[}FeiCarr{]}. In fact, we fully explore properties of
the wave solutions of the classical K.P.P (\cite{Sattinger}) equation
and the monotonic structure of system \eqref{eq:1.04}. For the existence
of the traveling wave solutions to the Lotka Volterra systems with
different assumptions on parameters, we refer to \cite{Hosono,Kanel,KanelZhou,Kan-on,TangFife}
and the references therein. Noting in the above mentioned results
little attention has been paid to the asymptotics of the wave solutions.
However, such information is the key to the understanding of the other
properties of the traveling wave solutions such as the strict monotonicity,
the uniqueness as well as the stability. As a final remark we point
out that the existence and stability of the traveling wave solutions
for \eqref{eq:1.02} is investigated in \cite{LeungHou} under conditions
H1 and H2 (with the inequality reversed).

\section{\label{sec:2}\textbf{The proof}}

\setcounter{equation}{0}

\noindent The proof of the Theorem is devided into several parts.

\subsection{\label{sub:2.1}The existence.}

We show the existence of the traveling wave solutions by monotone
iteration method given by \cite{WuZou}. Such method reduces the existence
of the wave solutions to the finding of an ordered pair of upper and
lower solutions. The construction of the upper and lower solutions
seems to be new, see also \cite{LeungHou}. 
\begin{defn}
\label{def:2}A continuous and essentially bounded function $(\bar{u}(\xi),\bar{v}(\xi))$,
$\xi\in\mathbb{R}$ is an upper solution of \eqref{eq:1.04} if it
satisfies \begin{equation}
\left\{ \begin{array}{l}
u''-cu'+u(1-u-a_{1}+a_{1}v)\leq0,\\
\\v''-cv'+r(1-v)(a_{1}u-v)\leq0,\end{array}\right.\mbox{ for }\xi\in\mathbb{R}/\left\{ y_{1,}y_{2},...y_{m}\right\} \label{eq:2.01}\end{equation}
and the boundary conditions \begin{equation}
\left(\begin{array}{c}
u\\
v\end{array}\right)(-\infty)\geq\left(\begin{array}{c}
0\\
0\end{array}\right),\,\;\left(\begin{array}{c}
u\\
v\end{array}\right)(+\infty)\geq\left(\begin{array}{c}
1\\
1\end{array}\right).\label{eq:2.02}\end{equation}
while at $y_{i}$, $i=1,2..m$, $m\in\mathbb{N}$, 

\begin{equation}
(\bar{u}'(y_{i}-0),\bar{v}'(y_{i}-0))^{T}\geq(\bar{u}'(y_{i}+0),\bar{v}'(y_{i}+0))^{T}.\label{eq:2.02-1}\end{equation}

A lower solution of \eqref{eq:2.01} is defined similarly by reversing
the above inequalities in \eqref{eq:2.01}, \eqref{eq:2.02} and \eqref{eq:2.02-1}. 
\end{defn}
First recall the following classical result (\cite{Sattinger}) on
the traveling wave solutions of K.P.P equation \begin{equation}
\left\{ \begin{array}{l}
w''-cw'+f(w)=0,\\
\\w(-\infty)=0,\quad w(+\infty)=b.\end{array}\right.\label{eq:2.03}\end{equation}
where $f\in C^{2}([0,\, b])$ and $f>0$ on the open interval $(0,b)$
with $f(0)=f(b)=0$, $f'(0)=d_{1}>0$ and $f'(b)=-d_{2}<0$. The Lemma
below describes the properties of the wave solurtions of \eqref{eq:2.03}.
\begin{lem}
\label{lem:2}Corresponding to every fixed wave speed $c\geq2\sqrt{d_{1}}$,
system \eqref{eq:2.03} has a unique (up to a translation of the origin)
strictly monotonically increasing traveling wave solution $w(\xi)$
for $\xi\in\mathbb{R}$. The traveling wave solution $w$ has the
following asymptotic behaviors:

For the wave solution with non-critical speed $c>2\sqrt{d_{1}}$,
we have \begin{equation}
w(\xi)=a_{w}e^{\frac{c-\sqrt{c^{2}-4d_{1}}}{2}\xi}+o(e^{\frac{c-\sqrt{c^{2}-4d_{1}}}{2}\xi})\mbox{ as }\xi\rightarrow-\infty,\label{eq:2.04}\end{equation}
 \begin{equation}
w(\xi)=b-b_{w}e^{\frac{c-\sqrt{c^{2}+4d_{2}}}{2}\xi}+o(e^{\frac{c-\sqrt{c^{2}+4d_{2}}}{2}\xi})\mbox{ as }\xi\rightarrow+\infty,\label{eq:2.05}\end{equation}
 where $a_{w}$ and $b_{w}$ are positive constants.

For the wave with critical speed $c=2\sqrt{d_{1}}$, we have\begin{equation}
w(\xi)=(a_{c}+d_{c}\xi)e^{\sqrt{d_{1}}\xi}+o(\xi e^{\sqrt{d_{1}}\xi})\mbox{ as }\xi\rightarrow-\infty,\label{eq:2.06}\end{equation}
\begin{equation}
w(\xi)=b-b_{c}e^{(\sqrt{d_{1}}-\sqrt{d_{1}+d_{2}})\xi}+o(e^{(\sqrt{d_{1}}-\sqrt{d_{1}+d_{2}})\xi})\mbox{ as }\xi\rightarrow+\infty,\label{eq:2.07}\end{equation}
 where the constants $d_{c}$ is negative, $b_{c}$ is positive and
$a_{c}\in\mathbb{R}$.

\medskip{}

\end{lem}
According to Lemma \ref{lem:2}, we let $c\geq2\sqrt{1-a_{1}}$ be
fixed and $\underline{u}(\xi)$, $\xi\in\mathbb{R}$ be a solution
of the following form of K.P.P equation

\begin{equation}
\left\{ \begin{array}{l}
g''(\xi)-g'(\xi)+(1-a_{1})g(\xi)(1-g(\xi))=0,\\
\\g(-\infty)=0,\, g(+\infty)=1,\end{array}\quad\quad\xi\in\mathbb{R}\right.\label{eq:2.08}\end{equation}
and for the same $c$ let $l$ be a number such that $\frac{r(a_{2}-1)-(1-a_{1})}{1-a_{1}+r}\leq l<1-a_{1}$
and $\bar{\bar{u}}(\xi)$, $\xi\in\mathbb{R}$ be a solution of a
K.P.P equation with the form \begin{equation}
\left\{ \begin{array}{l}
h''(\xi)-ch'(\xi)+(1-a_{1})h(\xi)(1-\frac{1-a_{1}-l}{1-a_{1}}h(\xi))=0,\\
\\h(-\infty)=0,\quad h(+\infty)=\frac{1-a_{1}}{1-a_{1}-l},\end{array}\quad\quad\xi\in\mathbb{R}.\right.\label{eq:2.09}\end{equation}

Define

\begin{equation}
\left(\begin{array}{l}
\bar{u}(\xi)\\
\\\bar{v}(\xi)\end{array}\right)=\left(\begin{array}{l}
\min_{\xi\in\mathbb{R}}\{\bar{\bar{u}}(\xi),1\}\\
\\\min_{\xi\in\mathbb{R}}\{(1+l)\bar{\bar{u}}(\xi),1\}\end{array}\right),\quad\left(\begin{array}{l}
\underline{u}(\xi)\\
\\\underline{v}(\xi)\end{array}\right)=\left(\begin{array}{l}
\underline{u}(\xi)\\
\\\underline{u}(\xi)\end{array}\right).\label{eq:2.10}\end{equation}

We have the following
\begin{lem}
\label{lem:3}For every fixed $c\geq2\sqrt{1-a_{1}}$, $(\bar{u}(\xi),\bar{v}(\xi))^{T}$
and $(u(\xi),\underbar{v}(\xi))^{T}$ in \eqref{eq:2.10} define respectively
an upper and lower solutions for \eqref{eq:1.04}-\eqref{eq:1.05}.\end{lem}
\begin{proof}
The verification for $(\underline{u},\underline{v}$) being a lower
solution for \eqref{eq:1.04}-\eqref{eq:1.05} is straightforward,
so we skip it. According to Lemma \ref{lem:2} $\bar{\bar{u}}(\xi)$
is strictly monotonically increasing for $\xi\in\mathbb{R}$, there
exist $N_{1},N_{2}\in\mathbb{R}$ such that $\bar{\bar{u}}(N_{1})=\frac{1}{1+l}$
and $\bar{\bar{u}}(N_{2})=1$. We can therefore rewrite $(\bar{u},\bar{v})$
as follows

\begin{equation}
\left(\begin{array}{l}
\bar{u}(\xi)\\
\\\bar{v}(\xi)\end{array}\right)=\left\{ \begin{array}{l}
(\bar{\bar{u}}(\xi),\bar{\bar{v}}(\xi))^{T},\quad\mbox{for }-\infty<\xi\leq N_{1};\\
\\(\bar{\bar{u}}(\xi),1)^{T},\quad\mbox{for }N_{1}<\xi\leq N_{2};\\
\\(1,1)^{T},\quad\mbox{for }N_{2}<\xi<+\infty.\end{array}\right.\label{eq:2.11-1}\end{equation}

For $N_{2}<\xi<+\infty$, $(\bar{u}(\xi),\bar{v}(\xi))^{T}=(1,1)^{T}$
is obviously a solution of \eqref{eq:1.04} and satisfying the inequality
\eqref{eq:2.02-1} on the boundary, we only need to verify that $(\bar{u}(\xi),\bar{v}(\xi))^{T}$
satisfies the inequality \eqref{eq:2.02-1} on the intervals $(-\infty,N_{1}]$
and $(N_{1},N_{2})$ respectively. For the pair $(\bar{u}(\xi),\bar{v}(\xi))^{T}=(\bar{\bar{u}}(\xi),(1+l)\bar{\bar{u}}(\xi))^{T}$
on $-\infty<\xi\leq N_{1}$, we have

\[
\begin{array}{ll}
 & \bar{\bar{u}}''-c\bar{\bar{u}}'+\bar{\bar{u}}(1-a_{1}-\bar{\bar{u}}+a_{1}\bar{\bar{v}})\\
\\= & -(1-a_{1})\bar{\bar{u}}(1-\frac{1-a_{1}-l}{1-a_{1}}\bar{\bar{u}})+\bar{\bar{u}}(1-a_{1}-\bar{\bar{u}}+a_{1}(1+l)\bar{\bar{u}})\\
\\= & -(1-a_{1})\bar{\bar{u}}(1-\frac{1-a_{1}-l}{1-a_{1}}\bar{\bar{u}}-1+\frac{1-a_{1}(1+l)}{1-a_{1}}\bar{\bar{u}})\\
\\= & -(1-a_{1})l\bar{\bar{u}}^{2}\leq0,\end{array}\]
and\[
\begin{array}{ll}
 & \bar{\bar{v}}''-c\bar{\bar{v}}'+r(1-\bar{\bar{v}})(a_{2}\bar{\bar{u}}-\bar{\bar{v}})\\
\\= & (1+l)(\bar{\bar{u}}''-c\bar{\bar{u}}'+\frac{r}{1+l}(1-(1+l)\bar{\bar{u}})(a_{2}\bar{\bar{u}}-(1+l)\bar{\bar{u}}))\\
\\= & (1+l)[\bar{\bar{u}}''-c\bar{\bar{u}}'+(1-a_{1})\bar{\bar{u}}(1-\bar{\bar{u}}+\frac{l}{1-a_{1}}\bar{\bar{u}})\\
\\ & -(1-a_{1})\bar{\bar{u}}(1-\bar{\bar{u}}+\frac{l}{1-a_{1}}\bar{\bar{u}})+\frac{r}{1+l}(1-(1+l)\bar{\bar{u}})\bar{\bar{u}}(a_{2}-(1+l))]\\
\\= & (1+l)\bar{\bar{u}}[\frac{r}{1+l}(1-(1+l)\bar{\bar{u}})(a_{2}-(1+l))-(1-a_{1})(1-\bar{\bar{u}}+\frac{l}{1-a_{1}}\bar{\bar{u}})]\\
\\= & (1+l)\bar{\bar{u}}[r(\frac{a_{2}}{1+l}-1)-(1-a_{1})-\bar{\bar{u}}(r(a_{2}-1-l)-(1-a_{1}-l)]\\
\\\leq & 0.\end{array}\]
The last inequality is true because of {[}\textbf{H3}{]} and the choice
of $l$.

For $N_{1}<\xi\leq N_{2}$, we verify that $(\bar{\bar{u}},1)^{T}$
satisfies the inequality \eqref{eq:2.01}. We only verify for the
first component since the one for $v=1$ is trivial.

\[
\begin{array}{ll}
 & \bar{\bar{u}}''-c\bar{\bar{u}}'+\bar{\bar{u}}(1-a_{1}-\bar{\bar{u}}+a_{1}v)\\
\\= & \bar{\bar{u}}''-c\bar{\bar{u}}'+\bar{\bar{u}}(1-\bar{\bar{u}})\\
\\ & +(1-a_{1})\bar{\bar{u}}(1-\frac{1-a_{1}-1}{1-a_{1}}\bar{\bar{u}})-(1-a_{1})\bar{\bar{u}}(1-\frac{1-a_{1}-1}{1-a_{1}}\bar{\bar{u}})\\
\\= & \bar{\bar{u}}(1-\bar{\bar{u}}-(1-a_{1})+(1-a_{1}-l)\bar{\bar{u}})\\
\\\leq & 0\end{array}\]

Therefore, we have the conclusion of the Lemma.
\end{proof}
To show the orderliness of the upper and lower solution pairs, we
first introduce a sliding domain method which applies to a sightly
more general system than \eqref{eq:1.02}. Noting that no monotonicity
requirements are imposed on the upper and lower sloutions. 
\begin{lem}
\label{lem:4}Let the $C^{2}$ vector functions $\bar{U}(\xi)=(\bar{u}_{1}(\xi),\bar{u}_{2}(\xi),...,\bar{u}_{n}(\xi))^{T}$
and $\underline{U}(\xi)=(\underline{u}_{1}(\xi),\underline{u}_{2}(\xi),...,\underline{u}_{n}(\xi))^{T}$
be $C^{2}$ and satisfy the following inequalities

\begin{equation}
D\bar{U}''-c\bar{U}'+F(U)\leq0\leq D\underline{U}''-c\underline{U}'+F(\underline{U})\quad\mbox{for\,}\,\xi\in[-N,N]\label{eq:2.11}\end{equation}
and

\begin{equation}
\underline{U}(-N)<\bar{U}(\xi)\quad\mbox{for\,}\,\xi\in(-N,N],\label{eq:2.12}\end{equation}

\begin{equation}
\underline{U}(\xi)<\bar{U}(N)\quad\mbox{for\,}\,\xi\in[-N,N),\label{eq:2.13}\end{equation}
where $D$ is a diagonal matrix with positive entries $D_{i}$, $i=1,2...n$,
$F(U)=(F_{1}(U),...,F_{n}(U))^{T}$ is $C^{1}$ with respect to its
components and $\frac{\partial F_{i}}{\partial u_{j}}\geq0$ for $i\neq j$,
$i,j=1,2...n$, then 

\begin{equation}
\underline{U}(\xi)\leq\bar{U}(\xi),\qquad\xi\in[-N,N].\label{eq:2.14}\end{equation}
\end{lem}
\begin{proof}
We adapt the proof of \cite{Berestycki} . Shift $\bar{U}(\xi)$ to
the left, for $0\leq\mu\leq2N$, consider $\bar{U}^{\mu}(\xi):=\bar{U}(\xi+\mu)$
on the interval $(-N-\mu,N-\mu)$. On both ends of the interval, by
\eqref{eq:2.12} and \eqref{eq:2.13}, we have

\begin{equation}
\underline{U}(\xi)<\bar{U}^{\mu}(\xi).\label{eq:2.15}\end{equation}
Starting from $\mu=2N$, decreasing $\mu$, for every $\mu$ in $0<\mu<2N$,
the inequality \eqref{eq:2.15} is true on the end points of the respective
interval. On decreasing $\mu$, suppose that there is a first $\mu$
with $0<\mu<2N$ such that

\[
\underline{U}(\xi)\leq\bar{U}^{\mu}(\xi)\quad\xi\in(-N-\mu,N-\mu)\]
and there is one component, for example the $i-th$, such that the
equality holds on a point $\xi_{1}$ inside the interval. Let $W(\xi)=(w_{1}(\xi),w_{2}(\xi),...,w_{n}(\xi))^{T}=\bar{U}^{r}(\xi)-\underline{U}(\xi)$,
then $w_{i}(\xi)$, $i=1,2,...,n$ satisfies 

\[
\left\{ \begin{array}{l}
D_{i}w_{i}''-cw_{i}'+\frac{\partial F_{i}}{\partial u_{i}}w_{i}\leq D_{i}w_{i}''-cw_{i}'+\Sigma_{j=1}^{n}\frac{\partial F_{i}}{\partial u_{j}}w_{j}\leq0,\\
\\w_{i}(\xi_{1})=0,\; w_{j}(\xi)\geq0\:\mbox{\, f\mbox{or}}\:\xi\in[-N-\mu,N-\mu],\end{array}\right.\]
the Maximum principle further implies that $w_{i}\equiv0$ for $\xi\in[-N-\mu,N-\mu]$,
but this is in contradiction with \eqref{eq:2.15} on the boundary
points $\xi=-N-\mu$ and $\xi=N-\mu$. So we can decrease $\mu$ all
the way to zero. This proves the Lemma.\end{proof}
\begin{lem}
\label{lem:5}There exists a $\nu\geq0$ such that $(\bar{u},\bar{v})^{T}(\xi+\nu)\geq(\underbar{u},\underbar{v})^{T}(\xi)\:\mbox{\, f\mbox{or}}\:\xi\in\mathbb{R}.$\end{lem}
\begin{proof}
We only prove for the wave speed $c>2\sqrt{1-a_{1}}$ and $r(a_{2}-1)\leq1$
since it is similar to show the other cases. We first derive the asymptotic
behaviors of the upper- and lower-solutions at infinities. By Lemma
\ref{lem:2}, we have the following asymptotics for the upper and
lower solutions \begin{equation}
\left(\begin{array}{c}
\bar{u}\\
\\\bar{v}\end{array}\right)(\xi)=\left(\begin{array}{c}
A_{1}\\
\\(1+l)A_{1}\end{array}\right)e^{\frac{c-\sqrt{c^{2}-4(1-a_{1})}}{2}\xi}+o(e^{\frac{c-\sqrt{c^{2}-4(1-a_{1})}}{2}\xi})\label{eq:2.16}\end{equation}
 and

\begin{equation}
\left(\begin{array}{c}
\underline{u}\\
\\\underline{v}\end{array}\right)(\xi)=\left(\begin{array}{c}
B_{1}\\
\\B_{1}\end{array}\right)e^{\frac{c-\sqrt{c^{2}-4(1-a_{1})}}{2}\xi}+o(e^{\frac{c-\sqrt{c^{2}-4(1-a_{1})}}{2}\xi})\label{eq:2.17}\end{equation}
 as $\xi\rightarrow-\infty$; 

\begin{equation}
\left(\begin{array}{c}
\bar{u}\\
\\\bar{v}\end{array}\right)(\xi)\equiv\left(\begin{array}{c}
1\\
\\1\end{array}\right)\label{eq:2.18}\end{equation}
 and

\begin{equation}
\left(\begin{array}{c}
\underline{u}\\
\\\underline{v}\end{array}\right)(\xi)=\left(\begin{array}{c}
1\\
\\1\end{array}\right)-\left(\begin{array}{c}
B_{2}\\
\\B_{2}\end{array}\right)e^{\frac{c-\sqrt{c^{2}+4r(a_{2}-1)}}{2}\xi}+o(e^{\frac{c-\sqrt{c^{2}+4r(a_{2}-1)}}{2}\xi})\label{eq:2.19}\end{equation}
 as $\xi\rightarrow+\infty$, where $A_{1}$, $A_{2}$, $B_{1}$,
\textbf{$B_{2}$} are positive constants.

Since (\ref{eq:2.06}) is translation invariant, $\bar{v}^{\tilde{r}}(\xi)\equiv\bar{v}(\xi+\tilde{r})$
is also a solution of (\ref{eq:2.06}) for any $\tilde{r}\in\mathbb{R}$.
It then follows that $(\bar{u}^{\tilde{r}},\bar{v}^{\tilde{r}})^{T}(\xi)$
is also an upper-solution pair for system \eqref{eq:1.04}-\eqref{eq:1.05}.
For the asymptotic behavior of $(\bar{u},\bar{v})^{\tilde{r}}(\xi)$
at $-\infty$, we can simply replace $(A_{1},\:(1+l)A_{1})^{T}$ by
$(A_{1},\:(1+l)A_{1})^{T}e^{\frac{c-\sqrt{c^{2}-4(1-a_{1})}}{2}\,\tilde{r}}$
in \eqref{eq:2.16}. Now we choose $\tilde{r}>0$ large enough such
that

\[
e^{\frac{c-\sqrt{c^{2}-4(1-a_{1})}}{2}\,\tilde{r}}>1.\]
 Then there exists a sufficiently large $N_{1}>0$ such that

\begin{equation}
\left(\begin{array}{c}
\bar{u}^{\tilde{r}}(\xi)\\
\\\bar{v}^{\tilde{r}}(\xi)\end{array}\right)>\left(\begin{array}{c}
\underline{u}(\xi)\\
\\\underline{v}(\xi)\end{array}\right)\quad{\normalcolor \mbox{for}\;}\xi\in(-\infty,-N_{1}].\label{eq:2.20}\end{equation}

On the other hand, the boundary conditions of the upper- and lower-solutions
at $+\infty$ also imply, on increasing $N_{1}$ if necessary, that

\begin{equation}
\left(\begin{array}{c}
\bar{u}^{\tilde{r}}(\xi)\\
\\\bar{v}^{\tilde{r}}(\xi)\end{array}\right)>\left(\begin{array}{c}
\underline{u}(\xi)\\
\\\underline{v}(\xi)\end{array}\right)\quad{\normalcolor \mbox{for}\;}\xi\in[N_{1},\:+\infty).\label{eq:2.21}\end{equation}

On the interval $[-N_{1},N_{1}],$ the strict monotonicity of the
upper and lower solutions $(\bar{u}^{\tilde{r}},\bar{v}^{\tilde{r}})^{T}$
and $(\underline{u},\underline{v})^{T}$, and the inequalities \eqref{eq:2.20}-\eqref{eq:2.21}
imply that 

\[
(\underline{u},\underline{v})^{T}(-N_{1})<(\bar{u}^{\tilde{r}},\bar{v}^{\tilde{r}})^{T}(\xi)\quad\mbox{for}\,\xi\in(-N_{1},N_{1}],\]
and

\[
(\underline{u},\underline{v})^{T}(\xi)<(\bar{u}^{\tilde{r}},\bar{v}^{\tilde{r}})^{T}(N_{1})\quad\mbox{for}\,\xi\in[-N_{1},N_{1}).\]
Therefore, by Lemma \ref{lem:4} we have

\begin{equation}
(\underline{u},\underline{v})^{T}(\xi)\leq\bar{(\bar{u}^{\tilde{r}},\bar{v}^{\tilde{r}})^{T}(\xi)}\quad\mbox{for}\,\xi\in[-N_{1},N_{1}].\label{eq:2.22}\end{equation}

Inequality \eqref{eq:2.22} along with \eqref{eq:2.20}, \eqref{eq:2.21}
show the validity of the Lemma.
\end{proof}
\textbf{Proof of the Existence:} We still use $(\bar{u},\bar{v})^{T}(\xi)$
to denote the shifted upper-solution as given in lemma \ref{lem:4}.
Applying the monotone iteration method given in \cite{WuZou} to the
upper and lower solutions defined in \eqref{eq:2.10}, we then have
the existence of the traveling wave solutions for $c\geq2\sqrt{a_{1}-1}$.
The boundary conditions that the upper and lower solutions satisfying
lead to the boundary conditions \eqref{eq:1.05} for traveling waves. 

\medskip{}

\subsection{\label{sub:2.2}The asymptotics and the monotonicity.}

\medskip{}

To derive the asymptotic decay rate of the traveling wave solution
at $\pm\infty$, we let $c\geq2\sqrt{1-a_{1}}$ and \begin{equation}
U(\xi):=(u(\xi),v(\xi))^{T}\quad\mbox{for}\:-\infty<\xi<\infty\label{eq:2.23}\end{equation}
be the corresponding traveling wave solution of \eqref{eq:1.04}-\eqref{eq:1.05}
resulted from the monotone iteration. Lemma \ref{lem:2} implies that
the upper- and the lower-solutions as derived in Lemma \ref{lem:3}
have the same asymptotic rates at $-\infty$. \eqref{eq:1.06} and
\eqref{eq:1.08} then follow from Lemma \ref{lem:5}. We differentiate
\eqref{eq:1.04} with respect to $\xi$, and note that $(U(\xi))':=(w_{1},w_{2})^{T}(\xi)$
satisfies \begin{equation}
(w_{1})_{\xi\xi}-c(w_{1})_{\xi}+A_{11}(u,v)w_{1}+A_{12}(u,v)w_{2}=0,\label{eq:2.24}\end{equation}
 \begin{equation}
(w_{2})_{\xi\xi}-c(w_{2})_{\xi}+A_{21}(u,v)w_{1}+A_{22}(u,v)w_{2}=0,\label{eq:2.25}\end{equation}
 where

\[
\begin{array}{ll}
A_{11}(u,v)=1-a_{1}-2u+a_{1}v,\quad & A_{12}(u,v)=a_{1}u\\
\\A_{21}(u,v)=a_{2}r(1-v), & A_{22}(u,v)=-r(a_{2}u+1-2v)\end{array}\]

We next study the exponential decay rates of the traveling wave solution
$U(\xi)$ at $+\infty$. The asymptotic system of (\ref{eq:2.24})
and (\ref{eq:2.25}) as $\xi\rightarrow+\infty$ is \begin{equation}
\left\{ \begin{array}{l}
(\psi_{1})''-c(\psi_{1})'-\psi_{1}+a_{1}\psi_{2}=0,\\
\\(\psi_{2})''-c(\psi_{2})'-r(a_{2}-1)\psi_{2}=0.\end{array}\right.\label{eq:2.26}\end{equation}

It is easy to see that the system \eqref{eq:2.26} admits exponential
dichotomy. Since the traveling wave solution $(u(\xi),v(\xi))^{T}$
converges monotonically to a constant limit as $\xi\rightarrow\pm\infty$,
the derivative of the traveling wave solution satisfies $(w_{1}(\pm\infty),w_{2}(\pm\infty))=(0,0)$
(\cite{WuZou}, p$658$ Lemma $3.2$). Hence we are only interested
in finding exponentially decaying solutions of \eqref{eq:2.26} at
$+\infty$.

One can write the the general solution of the second equation of \eqref{eq:2.26}
as \[
\psi_{2}=A^{1}e^{\frac{c+\sqrt{c^{2}+4r(a_{2}-1)}}{2}\xi}+B^{1}e^{\frac{c-\sqrt{c^{2}+4r(a_{2}-1)}}{2}\xi}\]
 for some constants $A^{1}$ and $B^{1}$. Since $w_{2}\rightarrow0$
as $\xi\rightarrow+\infty$, one immediately has $A^{1}=0$.

We then study the solution of the second equation of (\ref{eq:2.26}),
rewriting the equation as \begin{equation}
(\psi_{1})''-c(\psi_{1})'-\psi_{1}=-a_{1}\psi_{2},\label{eq:2.27}\end{equation}
we have the following expression for the solution of \eqref{eq:2.27},
\begin{equation}
\psi_{1}=\bar{B}_{1}e^{\frac{c-\sqrt{c^{2}+4r(a_{2}-1)}}{2}\xi}+\bar{B}_{2}e^{\frac{c-\sqrt{c^{2}+4}}{2}\xi}+\bar{B}_{3}e^{\frac{c+\sqrt{c^{2}+4}}{2}\xi}.\label{eq:2.28}\end{equation}

Since $w_{2}(\xi)\rightarrow0$ as $\xi\rightarrow+\infty$, then
$\bar{B}_{3}=0.$ Also noticing that (\ref{eq:2.27}) is non-homogeneous,
we have $\bar{B}_{1}\neq0$. By roughness of the exponential dichotomy
(levinson) and integration, we obtain the asymptotic decay rate of
the traveling wave solutions at $+\infty$ given in (\ref{eq:1.10}).

We next show the monotonicity of the traveling wave solutions. By
the monotone iteration process \cite{WuZou}, the traveling wave solution
$U(\xi)$ is increasing for $\xi\in\mathbb{R}$, it then follows that
$W(\xi)=U'(\xi)\geq0$ satisfying \eqref{eq:2.24} and \eqref{eq:2.25}
and

\begin{equation}
w_{1},w_{2}\geq0,\;(w_{1},w_{2})^{T}(\pm\infty)=0.\label{eq:2.31}\end{equation}
 The Maximum Principle implies that $(w_{1},w_{2})^{T}(\xi)>0$ for
$\xi\in\mathbb{R}$. This concludes that the traveling wave solution
is strictly increasing on $\mathbb{R}$.

\subsection{The Uniqueness.}

On the uniqueness of the traveling wave solution for every $c\geq2\sqrt{1-a_{1}}$,
we only prove the conclusion for traveling wave solutions with asymptotic
behaviors \eqref{eq:1.06} and \eqref{eq:1.07}, since other case
can be proved similarly. Let $U_{1}(\xi)$ and $U_{2}(\xi)$ be two
traveling wave solutions of system \eqref{eq:1.04}-\eqref{eq:1.05}
with the same speed $c>2\sqrt{1-a_{1}}$. There exist positive constants
$A_{i}$, $B_{i}$, $i=1,2,3,4$ and a large number $N>0$ such that
for $\xi<-N$,\begin{equation}
U_{1}(\xi)=\left(\begin{array}{c}
(A_{1}+o(1))e^{\frac{c-\sqrt{c^{2}-4(1-a_{1})}}{2}\xi}\\
\\(A_{2}+o(1))e^{\frac{c-\sqrt{c^{2}-4(1-a_{1})}}{2}\xi}\end{array}\right)\label{eq:2.32}\end{equation}
 \begin{equation}
U_{2}(\xi)=\left(\begin{array}{c}
(A_{3}+o(1))e^{\frac{c-\sqrt{c^{2}-4(1-a_{1})}}{2}\xi}\\
\\(A_{4}+o(1))e^{\frac{c-\sqrt{c^{2}-4(1-a_{1})}}{2}\xi}\end{array}\right);\label{eq:2.33}\end{equation}
 and for $\xi>N$,\begin{equation}
U_{1}(\xi)=\left(\begin{array}{c}
{\displaystyle 1-(B_{1}+o(1))e^{\frac{c-\sqrt{c^{2}+4r(a_{2}-1)}}{2}\xi}}\\
\\{\displaystyle 1-(B_{2}+o(1))e^{\frac{c-\sqrt{c^{2}+4r(a_{2}-1)}}{2}\xi}}\end{array}\right)\label{eq:2.34}\end{equation}
 \begin{equation}
U_{2}(\xi)=\left(\begin{array}{c}
{\displaystyle 1-(B_{3}+o(1))e^{\frac{c-\sqrt{c^{2}+4r(a_{2}-1)}}{2}\xi}}\\
\\{\displaystyle 1-(B_{4}+o(1))e^{\frac{c-\sqrt{c^{2}+4r(a_{2}-1)}}{2}\xi}}\end{array}\right)\label{eq:2.35}\end{equation}
 The traveling wave solutions of system \eqref{eq:1.04}\eqref{eq:1.05}
are translation invariant, thus for any $\theta>0$, $U_{1}^{\theta}(\xi):=U_{1}(\xi+\theta)$
is also a traveling wave solution of \eqref{eq:1.04} and \eqref{eq:1.05}.
By \eqref{eq:2.32} and \eqref{eq:2.34}, the solution $U_{1}(\xi+\theta)$
has the following asymptotic behaviors:\begin{equation}
U_{1}^{\theta}(\xi)=\left(\begin{array}{c}
(A_{1}+o(1))e^{\frac{c-\sqrt{c^{2}-4(1-a_{1})}}{2}\theta}e^{\frac{c-\sqrt{c^{2}-4(1-a_{1})}}{2}\xi}\\
\\(A_{2}+o(1))e^{\frac{c-\sqrt{c^{2}-4(1-a_{1})}}{2}\theta}e^{\frac{c-\sqrt{c^{2}-4(1-a_{1})}}{2}\xi}\end{array}\right)\label{eq:2.36}\end{equation}
 for $\xi\leq-N$;\begin{equation}
U_{1}^{\theta}(\xi)=\left(\begin{array}{c}
{\displaystyle 1-(B_{1}+o(1))e^{\frac{c-\sqrt{c^{2}+4r(a_{2}-1)}}{2}\theta}e^{\frac{c-\sqrt{c^{2}+4r(a_{2}-1)}}{2}\xi}}\\
\\{\displaystyle 1-(B_{2}+o(1))e^{\frac{c-\sqrt{c^{2}+4r(a_{2}-1)}}{2}\theta}e^{\frac{c-\sqrt{c^{2}+4r(a_{2}-1)}}{2}\xi}}\end{array}\right)\label{eq:2.37}\end{equation}
 for $\xi\geq N$.

It is clear that for $\theta$ large enough, we have \begin{equation}
A_{1}e^{\frac{c-\sqrt{c^{2}-4(1-a_{1})}}{2}\theta}>A_{3},\label{eq:2.38}\end{equation}
 \begin{equation}
A_{2}e^{\frac{c-\sqrt{c^{2}-4(1-a_{1})}}{2}\theta}>A_{4},\label{eq:2.39}\end{equation}
 \begin{equation}
B_{1}e^{\frac{c-\sqrt{c^{2}+4r(a_{2}-1)}}{2}\theta}<B_{3},\label{eq:2.40}\end{equation}
 \begin{equation}
B_{2}e^{\frac{c-\sqrt{c^{2}+4r(a_{2}-1)}}{2}\theta}<B_{4}.\label{eq:2.41}\end{equation}
 Inequalities \eqref{eq:2.38} - \eqref{eq:2.41} imply that for $\theta$
large enough, \begin{equation}
U_{1}^{\theta}(\xi)>U_{2}(\xi)\label{eq:2.42}\end{equation}
 for $\xi\in(-\infty,-N]$$\cup$$[N+\infty).$ We now consider $U_{1}^{\theta}(\xi)$
and $U_{2}(\xi)$ on $[-N,+N]$. 

On the interval $[-N_{1},N_{1}],$ the strict monotonicity of $U_{1}$
and $U_{2}$ and the inquality \eqref{eq:2.42} imply that 

\[
U_{2}(-N_{1})<U_{1}^{\theta}(\xi)\quad\mbox{for}\,\xi\in(-N_{1},N_{1}],\]
and

\[
U_{2}(\xi)<U_{1}^{\theta}(N_{1})\quad\mbox{for}\,\xi\in[-N_{1},N_{1}).\]
Therefore, by Lemma \ref{lem:4} we have

\begin{equation}
U_{1}^{\theta}(\xi)\geq U_{2}(\xi)\quad\mbox{for}\,\xi\in[-N_{1},N_{1}].\label{eq:2.43}\end{equation}

Inequalities \eqref{eq:2.43} along with \eqref{eq:2.42} show the
validity of the Lemma.

\subsection{The range of the wave speed.}

The next Theorem shows that the lower bound $2\sqrt{1-a_{1}}$ for
the wave speed is optimal, hence $c=2\sqrt{\alpha}$ is the critical
wave speed.
\begin{lem}
\label{lem:6-1}There is no monotone traveling wave solution of \eqref{eq:1.04}-\eqref{eq:1.05}
for any $0<c<2\sqrt{\alpha}$. \end{lem}
\begin{proof}
Suppose there is a constant $c$ with $0<c<2\sqrt{1-a_{1}}$ and a
solution $V(\xi)=(v_{1},v_{2})^{T}(\xi)$ of (\ref{eq:1.04}-\ref{eq:1.05})
corresponding to it. Similar to \ref{sub:2.2}, the asymptotic behaviors
of $V(\xi)$ at $-\infty$ are described by \[
\left(\begin{array}{c}
v_{1}(\xi)\\
\\v_{2}(\xi)\end{array}\right)=\left(\begin{array}{c}
A_{s}\\
\\B_{s}\end{array}\right)e^{\frac{c-\sqrt{c^{2}-4(1-a_{1})}}{2}\xi}+\left(\begin{array}{c}
\bar{A_{s}}\\
\\\bar{B_{s}}\end{array}\right)e^{\frac{c+\sqrt{c^{2}-4(1-a_{1})}}{2}\xi}+h.o.t,\]
 where $(A_{s},B_{s})^{T}$and $(\bar{A_{s},}\bar{B_{s}})$ are not
both zero. The condition $0<c<2\sqrt{1-a_{1}}$ implies that $V(\xi)$
is oscillating. This concludes that any solution of \eqref{eq:1.04}-\eqref{eq:1.05}
with $c<2\sqrt{1-a_{1}}$ is not strictly monotonic.  \end{proof}


\begin{thebibliography}{11}
\bibitem[1]{Berestycki}H. Berestycki, L. Nirenberg, Travelling fronts
in cylinders, Ann. Inst. H. Poincaré, Anal. Non Lin. 9 (1992), pp.
497-572. 

\bibitem[2]{Coddington}E. Coddington and N. Levinson, Theory of Ordinary
Differential Equations, McGraw-Hill, 1955.

\bibitem[3]{Fei} N. Fei and J. Carr, Existence of travelling waves
with their minimal speed for a diffusing Lotka-Volterra system, Nonlinear
Analysis: Real World Applications, 4 (2003), 503-524.

\bibitem[4]{Hosono}Y. Hosono, Travelling waves for a diffusive Lotka-Volterra
competition model I: Singular Perturbations, Discrete and Continuous
Dynamical Systems-Series B, Vol. 3, No 1, (2003) pp.79-95.

\bibitem[5]{Kanel}J. I Kanel, On the wave front of a competition-diffusion
system in popalation dynamics, Nonlinear Analysis, 65, (2006) 301-320.

\bibitem[6]{KanelZhou}J. I Kanel, Li Zhou, Existence of wave front
solutions and estimates of wave speed for a competition-diffusion
system, Nonlinear Analysis, Theory, Methods \& Applications, 27, No.
5, (1996) 579-587.

\bibitem[7]{Kan-on}Y. Kan-on, Note on propagation speed of travelling
waves for a weakly coupled parabolic system, Nonlinear Analysis 44
(2001) 239-246.

\bibitem[8]{Kan-on}Y. Kan-on, Fisher wave fronts for the lotka-volterra
competition model with diffusion, Nonlinear Analysis, Theory, methods
\& Applications, 28 No. 1, (1997) 145-164.

\bibitem[9]{LeungHou}A. Leung, X. Hou and W. Feng, Traveling wave
solutions for the Lotka Volterra system revisited. Submitted.

\bibitem[9]{Sattinger}D. Sattinger, On the stability of traveling
waves, Adv. in Math., 22 (1976), 312-355.

\bibitem[10]{TangFife}M. M. Tang and P. C. Fife, Propagating fronts
for competing species equations with diffusion, Arch. Rat. Mech. Anal.,
(73) 1980, pp 69-77.

\bibitem[11]{WuZou} J. Wu and X. Zou, Traveling wave fronts of reaction-diffusion
systems with delay, J. Dynamics and Diff. Eq. 13 (2001), 651-687.
and Erratum to Traveling wave fronts of reaction-diffusion
systems with delays, J. Dynamics and Diff. Eq. 2 (2008),
531-533.
\end{thebibliography}
\end{document}